\documentclass[10pt,a4paper]{amsart}
\title{Boundary $C^*$-algebras of triangle geometries}
\author{Guyan Robertson}
\address{School of Mathematics and Statistics, Newcastle University, Newcastle upon Tyne NE1 7RU, England, UK}
\email{guyanrobertson@gmx.com}

\date{}

\usepackage{amsmath}
\usepackage{amscd}
\usepackage{amssymb}

\input{prepictex}
\input{pictex}
\input{postpictex}
\chardef\bslash=`\\ 

\makeatletter
\def\verbatim{\interlinepenalty\@M \@verbatim
  \leftskip\@totalleftmargin\advance\leftskip2pc
  \frenchspacing\@vobeyspaces \@xverbatim}
\makeatother
\newtheorem{theorem}{Theorem}[section]
\newtheorem{corollary}[theorem]{Corollary}
\newtheorem{lemma}[theorem]{Lemma}
\newtheorem{proposition}[theorem]{Proposition}

\theoremstyle{definition}

\newtheorem{definition}[theorem]{Definition}
\newtheorem{remark}[theorem]{Remark}

\newcommand{\cl}[1]{{\mathcal{#1}}}
\newcommand{\bb}[1]{{\mathbb{#1}}}
\newcommand{\fk}[1]{{\mathfrak{#1}}}

\newcounter{picture}

\DeclareMathOperator{\coker}{coker}

\DeclareMathOperator{\conv}{conv}

\newcommand{\1}{{\bf 1}}

\newcommand{\PGL}{{\text{\rm{PGL}}}}

\newcommand{\SL}{{\text{\rm{SL}}}}

\begin{document}

\begin{abstract}
Let $\Delta$ be a building of type $\widetilde A_2$ and order $q$, with maximal boundary $\Omega$.
Let $\Gamma$ be a group of type preserving automorphisms of $\Delta$ which acts regularly on the chambers of $\Delta$.
Then  the crossed product $C^*$-algebra $C(\Omega) \rtimes \Gamma$ is isomorphic to $M_{3(q+1)}  \otimes\cl O_{q^2}\otimes\cl O_{q^2}$, where $\cl O_n$ denotes the Cuntz algebra generated by $n$ isometries whose range projections sum to the identity operator.
\end{abstract}

\maketitle

\section{Introduction}

The boundary at infinity plays an important role in studying the action of groups on euclidean buildings
and other spaces of non-positive curvature. The motivation of this article is to determine what information about a group may be recovered
from the K-theory of the boundary crossed product $C^*$-algebra. The focus is on groups of automorphisms of euclidean buildings. Groups of automorphisms of $\widetilde A_2$ buildings which act regularly on the set of chambers have been studied by several authors \cite{ kmw1, kmw2, ro1, tit}.  The existence of groups which act regularly on the set of chambers (i.e. the existence of ``tight triangle geometries'') is proved in \cite[Theorem 3.3]{ro1}.

The main result is the following, where $M_n$ denotes the algebra of complex $n\times n$ matrices and
$\cl O_n$ is the Cuntz $C^*$-algebra which is generated by $n$ isometries on a Hilbert space whose range projections sum to $I$.

\begin{theorem}\label{mainboundaryaction}
Let $\Delta$ be a building of type $\widetilde A_2$ and order $q$, with maximal boundary $\Omega$.
Let $\Gamma$ be a group of type preserving automorphisms of $\Delta$ which acts regularly on the chambers of $\Delta$.
Then  the crossed product $C^*$-algebra $\fk A_\Gamma = C(\Omega) \rtimes \Gamma$ is isomorphic to $M_{3(q+1)}  \otimes\cl O_{q^2}\otimes\cl O_{q^2}$.
\end{theorem}

The theorem is proved by giving an explicit representation of $\fk A_\Gamma$ as a rank two Cuntz-Krieger algebra in the sense of \cite {rs2}. The K-theory of $\fk A_\Gamma$ is then determined explicitly as  $K_0(\fk A_\Gamma)\cong K_1(\fk A_\Gamma)) \cong \bb Z_{q^2-1}$
with the class $[\1]$ in $K_0(\fk A_\Gamma)$ corresponding to $3(q+1) \in \bb Z_{q^2-1}$. The classification theorem of E.~Kirchberg and N.C.~Phillips for p.i.s.u.n $C^*$ algebras \cite{ad} completes the proof.
It is notable that, for these groups, $C(\Omega) \rtimes \Gamma$ depends only on the order $q$ of the building and not on the group $\Gamma$. This is the first class of higher rank groups for which the boundary $C^*$-algebra has been computed exactly. There is a sharp contrast with the class of groups which act regularly on the vertex set of an $\widetilde A_2$ building, where the boundary $C^*$-algebra $\fk A_\Gamma$ frequently appears to determine the group $\Gamma$, according to explicit computations derived from \cite{rs3}. For example, the three torsion-free groups  $\Gamma< \PGL_3({\bb Q}_2)$ which act regularly on the vertex set of the corresponding building $\Delta$ are distinguished from each other by $K_0(\fk A_\Gamma)$. On the other hand, for $q=2$, there are four different groups satisfying the hypotheses of Theorem \ref{mainboundaryaction}, and hence having the same boundary $C^*$-algebra.

\section{Triangle geometries}

A locally finite euclidean building whose boundary at infinity is a spherical building of rank 2 is of type $\widetilde A_2$, $\widetilde B_2$ or $\widetilde G_2$. This article is concerned with buildings of type $\widetilde A_2$, also called triangle buildings. There is a close connection to projective geometry: the link of each vertex of an $\tilde A_2$ building $\Delta$ is the incidence graph of a finite projective plane of order $q$.
The maximal simplices of $\Delta$ are triangles, which are called {\em chambers}. Each vertex of $\Delta$ has a {\it type} $j \in \bb Z_3$, and each chamber has exactly one vertex of each type. Each edge of $\Delta$ lies on $q+1$ chambers. Orient each edge of $\Delta$ from its vertex of type $i$ to its vertex of type $i+1$.
\refstepcounter{picture}
\begin{figure}[htbp]
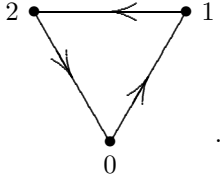
\label{chamber1}
\hfil
\centerline{
\beginpicture
\setcoordinatesystem units <1cm, 1.732cm>
\put {$\bullet$} at 0 0
\multiput {$\bullet$} at -1 1 *1 2 0 /
\put {$0$} [t] at 0 -0.1
\put {$2$} [r] at -1.2 1
\put {$1$} [l] at 1.2 1
\arrow <10pt> [.2, .67] from  0.2 1 to 0 1
\arrow <10pt> [.2, .67] from  -0.7 0.7 to -0.5 0.5
\arrow <10pt> [.2, .67] from  0.3 0.3 to 0.5 0.5
\putrule from -1 1 to 1 1
\setlinear \plot -1 1 0 0 1 1 /
\endpicture.
}
\hfil
\caption{A chamber showing vertex types and edge orientations.}
\end{figure}
In the link of the vertex $v$ of type $i$, the vertices of type $i+1$ [type $i+2$] correspond to the points [lines] of a projective plane
of order $q$. The chambers of $\Delta$ which contain $v$ correspond to incident point-line pairs. The number of such chambers is therefore
$(q+1)(q^2+q+1)$.

Suppose that $\Delta$ is a building of type $\widetilde A_2$ and that $\Gamma$ is a group of type preserving automorphisms of $\Delta$ which acts regularly (i.e. freely and transitively) on the chambers of $\Delta$. Then $\Delta /\Gamma$ is called a \textit{tight triangle geometry}. M. Ronan \cite[Theorem 3.3]{ro1} showed how to construct examples for all values of the prime power $q$. If $q=2$ then interesting examples arise from groups generated by three elements of order 3 such that any pair generate a Frobenius group of order 21 \cite{ro1, kmw1, kmw2}. These give rise to four non-isomorphic groups, the simplest of which has presentation
\begin{equation} \label{typeIgroup}
\big\langle s_i, i\in \bb Z_3 \, |\, s_i^3=1, s_is_{i+1}=(s_{i+1}s_i)^2 \big\rangle
\end{equation}
and acts on the euclidean building $\Delta$ of $\SL_3(\bb Q_2)$.
If $q=8$ there are 44 non-isomorphic groups of this type \cite {tit}.

From now on $\Delta$ will denote a building of type $\widetilde A_2$. Any two vertices~$u,v\in \Delta$ belong
to a common apartment. The convex hull, in the sense of buildings, between
two vertices $u$ and $v$ is illustrated in Figure~\ref{convex hulls}.
\refstepcounter{picture}
\begin{figure}[htbp]
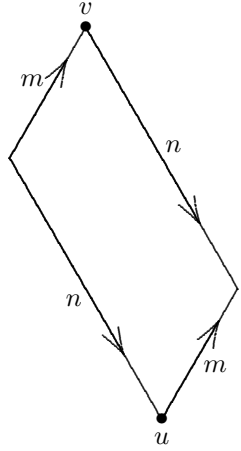
\label{convex hulls}
\hfil
\centerline{
\beginpicture
\setcoordinatesystem units <0.5cm,0.866cm>
\setplotarea x from -5 to 5, y from -2.5 to 4.5            
\put{$m$}  [tl]  at  2.1 -1.1
\put{$m$}  [br]  at -2.1  3.1
\put{$n$}  [tr]  at -1.1 -0.1
\put{$n$}  [bl]  at  1.1  2.1
\put{$\bullet$}  at  1   -2
\put{$u$}  [t]   at  1   -2.2
\put{$\bullet$}  at -1    4
\put{$v$}  [b]   at -1    4.2
\setlinear \plot 1 -2   3 0   -1 4   -3 2  1 -2 /
\arrow <10pt> [.3,.67] from   1    -2   to     2.5   -0.5
\arrow <10pt> [.3,.67] from  -3     2   to    -1.5   3.5
\arrow <10pt> [.3,.67] from  -1     4   to     2   1
\arrow <10pt> [.3,.67] from  -3     2   to     0  -1
\endpicture
}
\hfil
\caption{Convex hull of two vertices, with $\sigma (u,v) = (m,n)$.}
\end{figure}
Fundamental properties of buildings imply that any apartment containing~$u$
and~$v$ must also contain their convex hull.
Except in degenerate cases, the convex hull $\conv \{u, v\}$ is a parallelogram, with base vertex $u$.
Define the {\em distance}, $d(u,v)$, between vertices $u$ and $v$ to be the
graph theoretic distance on the one-skeleton of~$\Delta$. Any path from
$u$ to $v$ of length $d(u,v)$ lies in their convex hull, and the union
of the vertices in such paths is exactly the set of vertices in the
convex hull.

Define the {\em shape}~$\sigma(u,v)$ of an ordered pair of vertices
$(u,v)$ to be the pair $(m,n)\in\bb Z_+\times\bb Z_+$ as
indicated in Figure~\ref{convex hulls}.  Note that $d(u,v)=m+n$.  The arrows in Figure~\ref{convex hulls} point in the
direction of cyclically increasing vertex type, i.e.\
$\{\ldots,0,1,2,0,1,\ldots\}$.
If the parallelogram $\pi$ is the convex hull $\conv \{u, v\}$ of an {\em ordered} pair of vertices then the vertex $v_{ij}=v_{ij}(\pi)$ is the vertex of $\pi$ such that $\sigma(u, v_{ij})=(i,j)$. Parallelograms will always be parametrized in this way relative the initial vertex $u=v_{00}(\pi)$. Denote by $\Pi_{m,n}$ the set of parametrized parallelograms of shape $(m,n)$.

A \emph{sector} is a
$\frac{\pi}{3}$-angled sector made up of chambers in some apartment (Figure \ref{sector}).
Two sectors are  \emph{parallel} if their intersection contains a sector.
Parallelism is an equivalence relation and the (maximal) boundary $\Omega$ of $\Delta$ is defined to be the set of equivalence classes of sectors in $\Delta$.
For any vertex $v$  of $\Delta$ and any $\omega \in \Omega$ there is a unique sector $[v,\omega)$ in the
class $\omega$ having base vertex $v$ \cite[Section 9.3]{ron}.
Fix some vertex $O$ in $\Delta$.
The boundary of $\Delta$ is a totally disconnected compact Hausdorff space with a base for the
topology given by sets of the form
$$
\Omega(v) = \left \{ \omega \in \Omega : [O,\omega) \ \text{ contains } v \right \}
$$
where $v$ is a vertex of $\Delta$.

\refstepcounter{picture}
\begin{figure}[htbp]
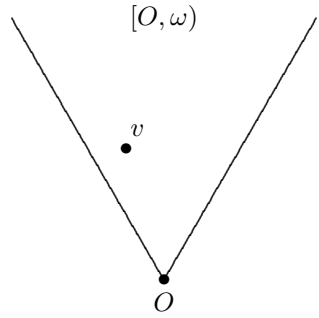
\label{sector}
\hfil
\centerline{
\beginpicture
\setcoordinatesystem units <0.5cm, 0.866cm>
\put {$\bullet$} at 0 0
\put {$\bullet$} at -1 2
\put {$[O,\omega)$}  at 0 4
\put {$O$} [t] at 0 -0.2
\put {$v$}  at  -0.7 2.3
\setlinear \plot -4 4  0 0  4 4 /
\endpicture
}
\hfil
\caption{The sector $[O,\omega)$, where $\omega\in \Omega(v)$.}
\end{figure}

A {\em tile} $\tau$ in $\Delta$ is defined to be the convex hull $\conv \{u, v\}$ of an ordered pair of vertices with shape $\sigma (u,v) = (2,2)$, as illustrated in Figure \ref{tile}. The initial vertex is $u=v_{00}(\tau)$ and the final vertex is $v=v_{22}(\tau)$.

\refstepcounter{picture}
\begin{figure}[htbp]
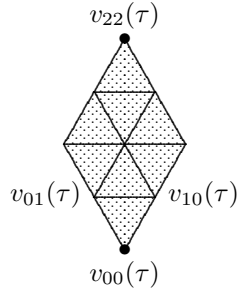
\label{tile}
\hfil
\centerline{
\beginpicture
\setcoordinatesystem units <0.4cm,0.7cm>  
\setplotarea x from -5 to 5, y from 0 to 4         
\setlinear \plot
 -1 1  0 0   2 2  0 4  -1 3  0 2  1 3 /
\plot -1 1  0 2  1 1  /
\plot  -1 1  -2 2  -1 3 /
\plot -1 1  1 1  / 
\plot  -2 2  2 2 /  
\plot -1 3  1 3  /   
\setshadegrid span <1.5pt>
\vshade  -2  2  2  <,z,,>  0  0  4  <z,,,>  2  2  2  /
\put{$\bullet$}  at  0   0
\put{$v_{00}(\tau)$}  [t]   at  0   -0.2
\put{$\bullet$}  at 0    4
\put{$v_{22}(\tau)$}  [b]   at 0    4.2
\put{$v_{10}(\tau)$}  [b]   at 2.6   0.8
\put{$v_{01}(\tau)$}  [b]   at -2.6  0.8
\endpicture
}
\hfil
\caption{A tile $\tau$ in an apartment.}
\end{figure}

Let $\pi$ be a parallelogram with shape $(m,n)$, where $m, n \ge 2$. The terminal tile $t(\pi)$ is the tile $\tau$ contained in the  parallelogram $\pi$ and containing its terminal vertex $v_{m,n}(\pi)$ (Figure \ref{parallelogram}).
Let
 $$\Omega(\pi)=\{ \omega\in \Omega\, : \, \pi\subset [v_{00}(\pi),\omega)\},$$
the set of boundary points whose representative sectors based at $v_{00}(\pi)$ contain $\pi$.

\refstepcounter{picture}
\begin{figure}[htbp]
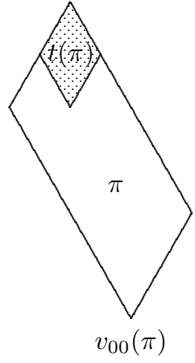
\label{parallelogram}
\hfil
\centerline{
\beginpicture
\setcoordinatesystem units <0.4cm,0.7cm>    
\setplotarea x from -6 to 6, y from -2 to 4         
\put{$v_{00}(\pi)$}  [t] at  1 -2.2
\put{$\pi$}  [t] at  0.5  0.6
\put{$t(\pi)$}  [r] at  -.2  3
\setlinear
\plot 1  -2    3  0   -1 4   -3 2   1 -2 /
\plot -2 3  -1 2   0 3 /
\setshadegrid span <1.5pt>
\vshade  -2  3  3  <,z,,>  -1  2  4  <z,,,>  0  3  3  /
\endpicture
}
\hfil
\caption{A parallelogram $\pi$ and its terminal tile $t(\pi)$.}
\end{figure}

\section{Algebras arising from boundary actions on $\widetilde A_2$ buildings}\label{A2boundary}

Let $\Delta$ be a building of type $\widetilde A_2$ and order $q$, with boundary $\Omega$.
Let $\Gamma$ be a group of type preserving automorphisms of $\Delta$ which acts regularly on the chambers of $\Delta$.
The group $\Gamma$ acts on $\Omega$, and one can form the crossed product $C^*$-algebra $C(\Omega) \rtimes \Gamma$.
This is the universal $C^*$-algebra  generated by the commutative $C^*$-algebra $C(\Omega)$ and the image of a unitary representation $\pi$ of $\Gamma$, satisfying the covariance relation
\begin{equation}\label{cov0}
f(\gamma^{-1}\omega) = \pi(\gamma)\cdot f \cdot \pi(\gamma)^{-1}(\omega)
\end{equation} for
$f \in C(\Omega)$, $\gamma \in \Gamma$ and $\omega\in\Omega$.
It is convenient to denote $\pi(\gamma)$ simply by $\gamma$.
We show that the crossed product algebra $\fk A_\Gamma = C(\Omega) \rtimes \Gamma$ is generated by an explicit set of partial isometries, described as follows. Fix a vertex $O$ with type $0$ in $\Delta$. Denote by $\Pi_O$ the set of all (parametrized) parallelograms in $\Delta$ with base vertex $O$. If $\pi_1, \pi_2 \in \Pi_O$ with $\Gamma t(\pi_1)= \Gamma t(\pi_2)$, then there is a  unique element $\gamma \in \Gamma$  such that $\gamma t(\pi_1)=t(\pi_2)$,
since $\Gamma$ acts freely on the chambers of $\Delta$. Let
\begin{equation}\label{Sab}
S_{\pi_2,\pi_1}=\gamma \bf 1 _{\Omega(\pi_1)}= \bf 1 _{\Omega(\pi_2)}\gamma.
\end{equation}
Then $S_{\pi_2,\pi_1}$ is a partial isometry with initial projection
$\bf 1 _{\Omega(\pi_1)}$
and final projection
$\bf 1 _{\Omega(\pi_2)}$.

\begin{lemma}\label{generate}
  The set of partial isometries
  $$\{S_{\pi_2,\pi_1} : \pi_1, \pi_2 \in \Pi_O, \Gamma t(\pi_1)= \Gamma t(\pi_2) \}$$
   generates $C(\Omega)\rtimes \Gamma$.
\end{lemma}

\begin{proof}
Let $\fk S$ denote the $C^*$-subalgebra of $\fk A_\Gamma$ generated by $\{S_{\pi_2,\pi_1} : \pi_1, \pi_2 \in \Pi_O, \Gamma t(\pi_1)= \Gamma t(\pi_2) \}$.
  Since the sets of the form $\Omega(\pi)$, $\pi\in\Pi_O$, form a basis for the topology of $\Omega$, the linear span of the $\Omega (\pi)$ is dense in $C(\Omega)$.
It follows that  $\fk S$  contains $C(\Omega)$.
To show that $\fk S$ contains $\fk A_\Gamma$, it suffices to show that it contains $\Gamma$.

Fix $\gamma\in\Gamma$. Choose $m_1,m_2 \ge 2$ such that $d(O,\gamma^{-1}O)\le m_1,m_2$.
Note that
\begin{equation}\label{sump}
\gamma=\sum_\pi \gamma{\bf 1}_{\Omega(\pi)}
\end{equation}
where the sum is over all parallelograms $\pi$ based at $O$ with $\sigma(\pi)=(m_1,m_2)$.
Fix such a parallelogram $\pi$. Let $\pi''=\conv\{\gamma^{-1}O,x\}$.
Since $\sigma(t(\pi))=(2,2)$, it follows from \cite [Lemma 7.5]{rs2} that $t(\pi'')=t(\pi)$.

\refstepcounter{picture}
\begin{figure}[htbp]
\hfil
\centerline{
\beginpicture
\setcoordinatesystem units <0.5cm,0.866cm>    
\setplotarea x from -6 to 6, y from -2.2 to 4.3         
\put{$O$}  [t] at  1  -2.2
\put{$x$}  [l] at  -0.8  4.2
\put{$\gamma^{-1}O$}  [t] at  -2.5  -1.7
\put{$t(\pi'')=t(\pi)$}  [r] at  -1.6 3.6
\setlinear
\plot 1  -2    3  0   -1  4     -3 2   1 -2   /
\plot -3 2    -4.5 0.5   -2.5 -1.5    1 2   /
\plot -1.5 3.5  -1 3  -0.5 3.5 /
\setshadegrid span <1.5pt>
\vshade  -1.5  3.5  3.5  <,z,,>  -1  3  4  <z,,,>  -0.5 3.5 3.5  /
\endpicture
}
\hfil
\caption{}
\end{figure}

Let $\pi'=\gamma\pi''$. Then $\pi'=\conv\{O,\gamma x\}$ and $t(\pi')=\gamma t(\pi)$.
Therefore $\gamma{\bf 1}_{\Omega(\pi)}=S_{\pi,\pi'} \in \fk A_\Gamma$.
The result follows from equation (\ref{sump}).
\end{proof}

\section{Cuntz-Krieger algebras of rank 2}\label{appendix}

This section recalls the principal facts about these algebras from \cite{rs2}. Fix a finite set $A$ which we refer to as an ``alphabet''.
A \emph{$\{0,1\}$-matrix} is a matrix with entries in $\{0,1\}$.
Choose nonzero  $\{0,1\}$-matrices  $M_1, M_2$ and denote their elements by
 $M_j(b,a) \in \{0,1\}$ for $a,b \in A$.

If $(m,n)\in \bb Z_+\times \bb Z_+$, a word of \textit{shape} $\sigma(w)=(m,n)$ in the alphabet $A$
is a mapping $(i,j)\mapsto w_{i,j}$ from $\{0, 1, \dots , m\} \times \{0, 1, \dots , n\}$ to $A$ such that
$M_1(w_{i+1,j}, w_{i,j}) = 1$ and $M_2(w_{i,j+1}, w_{i,j}) = 1$.
Denote by $W_{m,n}$ the set of words $w$ of shape $(m,n)$.
The set $W_{0,0}$ is identified with $A$.

Define initial and final maps $o: W_{m,n} \to A$ and $t: W_{m,n} \to A$ by
$o(w) = w_{0,0}$ and $t(w) = w_{m,n}$.

\refstepcounter{picture}
\begin{figure}[htbp]
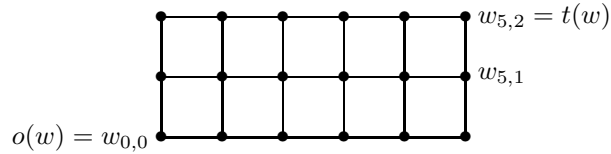
\label{word}
\hfil
\centerline{
\beginpicture
\setcoordinatesystem units <0.8cm, 0.8cm>
\setplotarea  x from -6 to 6,  y from -1 to 1.2
\putrule from -3 -1 to 2 -1
\putrule from -3 0 to 2 0
\putrule from -3 1  to 2  1
\putrule from  -3 -1  to -3 1
\putrule from  -2 -1  to -2 1
\putrule from  -1 -1  to -1 1
\putrule from  0 -1  to -0 1
\putrule from  1 -1  to 1 1
\putrule from  2 -1  to 2 1
\put {$w_{5,2}=t(w)$} [l]     at   2.2 1
\put {$w_{5,1}$} [l]     at   2.2 0
\put {$o(w)=w_{0,0}$} [r]     at   -3.2 -1
\multiput {$\bullet$} at -3 -1 *5  1 0 /
\multiput {$\bullet$} at -3 0 *5  1 0 /
\multiput {$\bullet$} at -3 1 *5  1 0 /
\endpicture
}
\hfil
\caption{Representation of a two dimensional word of shape $(5,2)$.}
\end{figure}

Fix a nonempty finite or countable set $D$ (whose elements are \textit{decorations}),
and a map $\delta : D \to A$.
Let $\overline W_{m,n} = \{ (d,w) \in D \times W_{m,n} : \ o(w) = \delta (d) \}$, the set of
\textit{decorated words} of shape ${m,n}$, and identify
$D$ with $\overline W_0$ via the map $d \mapsto (d,\delta(d))$.
Let $W = \bigcup_{m,n} W_{m,n}$ and  $\overline W = \bigcup_{m,n} \overline W_{m,n}$, the sets of all
words and all decorated words respectively.
Define $o: \overline W_{m,n} \to D$ and $t: \overline W_{m,n} \to A$ by
$o(d,w) = d$ and $t(d,w) = t(w)$. Likewise extend the definition of shape to $\overline W$ by setting
$\sigma((d,w))=\sigma(w)$.

If $M_1M_2=M_2M_1$ and $M_1M_2$ is a $\{0,1\}$-matrix then by \cite[Lemma 1.4]{rs3} the following condition holds.

\begin{description}
\item[(H1)]\label{H1} Let $u\in W_{m,n}$ and $v \in W_{r,s}$. If $t(u) =o(v)$ then there exists a unique
word $w\in W_{m+r, n+s}$ such that
\begin{eqnarray*}
  w_{i,j} &=& u_{i,j} \qquad  \qquad   0\le i \le m, 0 \le j \le n; \\
  w_{i,j} &=& v_{m+i,n+j} \qquad  0\le i \le r, 0 \le j \le s.
\end{eqnarray*}
\end{description}

\noindent  We write $w=uv$ and say that the product $uv$ exists.

\refstepcounter{picture}
\begin{figure}[htbp]
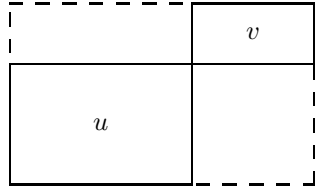
\label{productdiagram}
\centerline{
\beginpicture
\setcoordinatesystem units <0.8cm, 0.8cm>
\setplotarea  x from 0 to 5,  y from 0 to 3
\putrectangle corners at 0 0 and 3 2
\putrectangle corners at 3 2 and 5 3
\setdashes
\putrectangle corners at 0 0 and 5 3
\put{$u$} at 1.5  1
\put{$v$} at 4  2.5
\endpicture
}
\caption{The word  $uv$}
\end{figure}

\noindent Assume that condition \textbf{(H1)} is satisfied, as well as the following two conditions.

\begin{description}
\item[(H2)] Consider the directed graph which has a vertex for each $a \in A$
and a directed edge from $a$ to $b$ for each $i$ such that $M_i(b,a) =1$.
This graph is irreducible.

\item[(H3)] If $(p_1,p_2)\in \bb Z^2 -\{(0,0)\}$, then there exists some  $w \in W$ which is not $(p_1,p_2)$-periodic, in the sense of \cite{rs2}.
\end{description}

\noindent Under the above assumptions, define a $C^*$-algebra, which depends on $D$, as follows.

\begin{definition}\label{CKdefinition}
The $C^*$-algebra $\cl A_D$ is the universal $C^*$-algebra
generated by a family of partial isometries
$\{s_{u,v} : \ u,v \in \overline W \ \text{and} \ t(u) = t(v) \}$
satisfying the relations
\begin{subequations}\label{rel1*}
\begin{eqnarray}
{s_{u,v}}^* &=& s_{v,u} \label{rel1a*}\\
s_{u,v}s_{v,w}&=&s_{u,w} \label{rel1b*}\\
s_{u,v}&=&\displaystyle\sum_
{\substack{w\in W;\sigma(w)=\varepsilon,\\
o(w)=t(u)=t(v)}}
s_{uw,vw} ,\ \text{for}\quad  \varepsilon\in\{(1,0), (0,1)\}
\label{rel1c*}\\
s_{u,u}s_{v,v}&=&0 ,\ \text{for} \ u,v \in \overline W_0, u \ne v. \label{rel1d*}
\end{eqnarray}
\end{subequations}
\end{definition}

\begin{theorem}\label{mainck} \cite{rs2} The $C^*$-algebra $\cl A_D$ is purely infinite, simple and nuclear. Any nontrivial $C^*$-algebra with
generators $S_{u,v}$ satisfying relations (\ref{rel1*}) is isomorphic to $\cl A_D$.
\end{theorem}

\begin{remark}\label{decorate}
Two decorations $\delta_1 : D_1 \to A$ and $\delta_2 : D_2 \to A$ are said to be {\it equivalent} \cite[Section 5]{rs1} if there is a bijection $\eta: D_1 \to D_2$ such that $\delta_1 = \delta_2 \eta$. Equivalent decorations $\delta_1, \delta_2$ give rise to isomorphic algebras $\cl A_{D_1}, \cl A_{D_2}$.
\end{remark}

\begin{lemma}\label{decorate*}
Given a decoration $\delta : D \to A$, and $r\in \bb N$, define another decoration $\delta' : D \times \{1,2,\dots ,r\} : \to A$
by $\delta'((d,i)) = \delta(d)$. The $C^*$-algebra $\cl A_{D \times \{1,2,\dots ,r\}}$ is isomorphic to $M_r \otimes \cl A_D$.
\end{lemma}

\begin{proof}
(c.f. \cite [Lemma 5.12]{rs2}.)
If $u,v \in W$, the isomorphism is given by
$$s_{((d,i),u),((d',j),v)} \mapsto s_{(d,u),(d',v)}\otimes e_{ij},$$
where the $e_{ij}$ are matrix units for $M_r$. The fact that this
is an isomorphism follows from \cite[Corollary 5.10]{rs2}.
\end{proof}

\section{$\fk A_\Gamma = C(\Omega)\rtimes \Gamma$ as a rank 2 Cuntz-Krieger algebra}

We now show how the set of partial isometries $\{S_{\pi_2,\pi_1} : \pi_1, \pi_2 \in \Pi_O, \Gamma t(\pi_1)= \Gamma t(\pi_2) \}$ defined by (\ref{Sab}) allows one to express
$\fk A_\Gamma$ as a rank 2 Cuntz-Krieger algebra.
The alphabet $A$ is the set of $\Gamma$-orbits of tiles.
That is $A=\{\Gamma\tau:{\tau}\in \Pi_{2,2}\}$.  Since $\Gamma$ is type preserving,
$A=A_0 \sqcup A_1 \sqcup A_2$, where $A_i$ denotes the set of $\Gamma\tau$ such that $\tau$ has base vertex of type $i$.
The decorating set $D$ is the set of tiles $\tau$ with fixed base vertex $O$, and the decorating map
$\delta : D \to A$ is defined by $\delta(\tau)=\Gamma\tau$.
From now on, $D$ will always denote this particular decorating set.

The matrices $M_1$, $M_2$ with entries in $\{0,1\}$
are defined as follows.
If $a,b \in A$, say that $M_1(b,a)=1$ if and only if
$a=\Gamma \tau_1$, $b = \Gamma \tau_2$, for some tiles $\tau_1, \tau_2$, where the union of the representative
tiles $\tau_1 \cup \tau_2$ is a parallelogram of shape $(3,2)$ as illustrated
on the right of Figure \ref{transition}, with $\tau_1$ shaded.
The definition of $M_2$ is illustrated on the left of Figure \ref{transition},
where $a'=\Gamma \tau_1'$, $b' = \Gamma \tau_2'$ and $\tau_1' \cup \tau_2'$ is a parallelogram of shape $(2,3)$.

\refstepcounter{picture}
\begin{figure}[htbp]
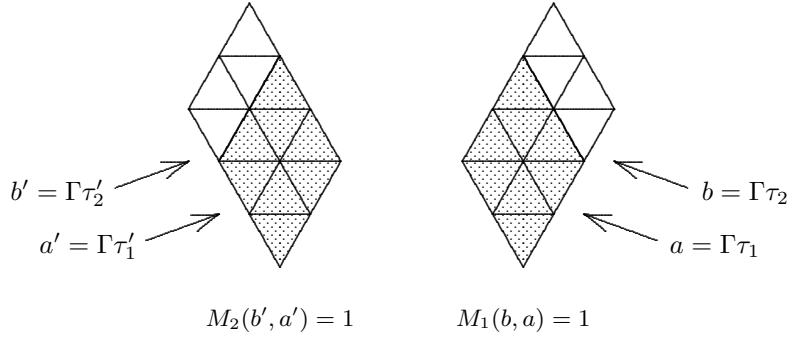
\label{transition}
\hfil
\centerline{
\beginpicture
\setcoordinatesystem units <0.4cm,0.7cm>  point at -4 0 
\setplotarea x from -5 to 5, y from 0 to 5         
\setlinear \plot
 -1 1  0 0   2 2  0 4  -1 3  0 2  1 3 /
\plot -1 1  0 2  1 1  /
\plot  -1 1  -2 2  -1 3 /
\plot -1 1  1 1  / 
\plot  -2 2  2 2 /  
\plot -1 3  1 3  /   
\plot 0 4   1 5  2 4  0 4  1 3  2 4  3 3  1 3  2 2  3 3 /  
\setshadegrid span <1.5pt>
\vshade  -2  2  2  <,z,,>  0  0  4  <z,,,>  2  2  2  /
\arrow <10pt> [.3,.67] from   5.4  1.5   to     3  2
\put{$b=\Gamma\tau_2$}  at 7.3  1.4
\arrow <10pt> [.3,.67] from   4.4  0.5   to     2  1
\put{$a=\Gamma\tau_1$}  at 6.3  0.4
\put{{\small $M_1(b,a)=1$}}   at  0 -1
\setcoordinatesystem units <0.4cm,0.7cm>  point at 4 0 
\setplotarea x from -5 to 5, y from 0 to 5         
\setlinear \plot
 -1 1  0 0   2 2  0 4  -1 3  0 2  1 3 /
\plot -1 1  0 2  1 1  /
\plot  -1 1  -2 2  -1 3 /
\plot -1 1  1 1  / 
\plot  -2 2  2 2 /  
\plot -1 3  1 3  /   
\plot 0 4   -1 5  -2 4  0 4  -1 3  -2 4  -3 3  -1 3  -2 2  -3 3 /  
\setshadegrid span <1.5pt>
\vshade  -2  2  2  <,z,,>  0  0  4  <z,,,>  2  2  2  /
\arrow <10pt> [.3,.67] from   -5.4  1.5   to     -3  2
\put{$b'=\Gamma\tau_2'$}  at -7.3  1.4
\arrow <10pt> [.3,.67] from   -4.4  0.5   to     -2  1
\put{$a'=\Gamma\tau_1'$}  at -6.3  0.4
\put{{\small $M_2(b',a')=1$}}   at  0 -1
\endpicture
}
\hfil
\caption{Transition matrices.}
\end{figure}

\begin{proposition}\label{M_1M_2} The matrices
$M_1$, $M_2$ satisfy conditions {\rm(H1)},{\rm(H2)}, and {\rm(H3)} of Section \ref{appendix}.
\end{proposition}

\begin{proof} The verification of conditions (H1) and (H2) is a minor modification of
\cite [Proposition 7.9]{rs2}, since $\Gamma$ acts freely on the chambers of $\Delta$, but the verification of (H2) requires a new argument.

Consider the directed graph $\cl G$ which has vertex set $A$
and a directed edge from $a \in A$ to $b \in A$ for each $i=1,2$ such that $M_i(b,a) =1$.
Condition (H2) of \cite{rs2} is the assertion that this graph is irreducible. In order to prove this, fix tiles $\tau_1, \tau_2$ such that
$a=\Gamma \tau_1, b=\Gamma \tau_2$. Fix a sector $\fk S$ containing $\tau_1$, with base point $v_{00}(\tau_1)$, as in Figure \ref {irreducible}.
Choose from among the shaded chambers in Figure \ref{irreducible} the chamber $c$ whose initial vertex has the same type as the initial vertex of
$\tau_2$. Since $\Gamma$ acts transitively on the chambers of $\Delta$, we may choose $\gamma\in\Gamma$ such that $\gamma^{-1}c$
 is the initial chamber of $\tau_2$. Then $\conv \{\tau_1 \cup \gamma\tau_2\}$ is contained in a sector $\fk S'$ with initial tile $\tau_1$. Figure \ref{irreducible'} illustrates one of the three possible relative positions.
 This shows that there is a directed path of length $3$ in the graph $\cl G$ from $\Gamma \tau_1$ to $\Gamma\tau_2$.
 \end{proof}

\refstepcounter{picture}
\begin{figure}[htbp]
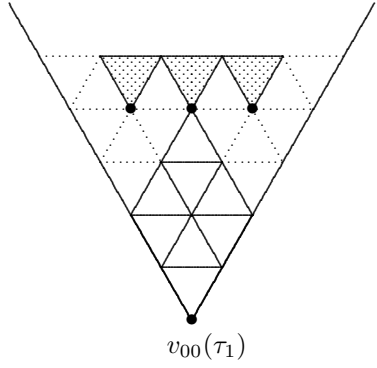
\label{irreducible}
\hfil
\centerline{
\beginpicture
\setcoordinatesystem units <0.4cm,0.7cm>  
\setplotarea x from -5 to 5, y from 0 to 6         
\setlinear \plot
 -1 1  0 0   2 2  0 4  -1 3  0 2  1 3 /
\plot -1 1  0 2  1 1  /
\plot  -1 1  -2 2  -1 3 /
\plot -1 1  1 1  / 
\plot  -2 2  2 2 /  
\plot -1 3  1 3  /   
\setshadegrid span <1.5pt>
\vshade  -3 5 5  <,z,,> -2 4 5 <z,z,,> -1 5 5  <z,z,,> 0 4 5  <z,z,,>  1 5 5 <z,z,,> 2 4 5  <z,,,>  3 5 5  /
\plot -6 6   0 0   6 6  / 
\plot  -3 5  -2 4  -1 5  0 4  1 5  2 4  3 5  -3 5 /       
\put{$v_{00}(\tau_1)$}  at  0.5   -0.5
\put{$\bullet$}  at  0   0
\put{$\bullet$}  at 0    4
\put{$\bullet$}  at 2    4
\put{$\bullet$}  at -2    4
\setdots <2.5pt>
\plot -3 3  -1 3  /   
\plot 1 3  3 3  /   
\plot -4 4  4 4  /   
\plot -5 5  -3 5  /   
\plot 3 5  5 5  /   
\plot -3 3  -2 4  -1 3  /   
\plot 1 3  2 4  3 3  /   
\plot -4 4   -3 5  /   
\plot 4 4   3 5  /   
\endpicture
}
\hfil
\caption{The sector $\fk S$ containing $\tau_1$.}
\end{figure}

\refstepcounter{picture}
\begin{figure}[htbp]
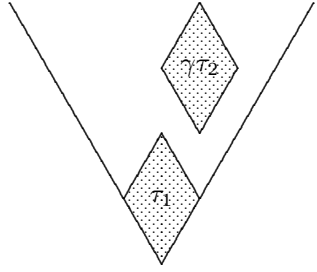
\label{irreducible'}
\hfil
\centerline{
\beginpicture
\setcoordinatesystem units <0.25cm,0.433cm>  
\put{$\tau_1$}  at  0  2
\put{$\gamma\tau_2$}  at  2  6
\setplotarea x from -5 to 5, y from 0 to 8         
\plot -2 2  0 4  2 2 /  
\plot  0 6  2 4  4 6  2 8  0 6 / 
\setshadegrid span <1.5pt>
\vshade  -2 2 2  <,z,,> 0 0 4 <z,,,> 2 2 2  /
\vshade  0 6 6  <,z,,> 2 4 8 <z,,,> 4 6 6 /
\plot -8 8   0 0   8 8  / 
\endpicture
}
\hfil
\caption{The sector $\fk S'$ containing $\tau_1$ and $\gamma\tau_2$.}
\end{figure}

The proof that $\{S_{\pi_2,\pi_1} : \pi_1, \pi_2 \in \Pi_O, \Gamma t(\pi_1)= \Gamma t(\pi_2) \}$ is the set of generators
for a rank 2 Cuntz-Krieger algebra $\cl A_D$, associated with the alphabet $A$ and decorating map $\delta : D \to A$ defined above, is a minor modification of the argument in Section 7 of \cite{rs2}. Let $\fk W_{m,n}=\Gamma\Pi_{m+2,n+2}$, the set of $\Gamma$-orbits of parallelograms of shape $(m+2,n+2)$. The elements of $\fk W_{m,n}$ are in bijective correspondence with the words of shape $(m,n)$ in the alphabet $A$.
See Section \ref{appendix}, and compare with \cite[Lemma 7.1]{rs2}.
Thus the large parallelograms on the right and left of Figure \ref{transition} represent words $ab$ and $a'b'$ of shape $(1,0)$
and $(0,1)$ respectively.

\begin{remark}
  The overlap in the tiles used to define the transition matrices $M_1$ and $M_2$ is needed
  in order to establish the bijection from $W_{m,n}$ onto $\fk W_{m,n}$. It is not possible to use the more straightforward definition
  of the transition matrices, in terms of contiguous, but not overlapping, tiles.
\end{remark}

In  view of Lemma \ref{generate}, we now have the following result.
\begin{proposition}\label{GammaD}
If $\fk A_\Gamma = C(\Omega) \rtimes \Gamma$, then there is an isomorphism
$\fk A_\Gamma\cong \cl A_D$.
\end{proposition}

The next result will be used later.

\begin{lemma}\label{01} Let $i\in\{0,1,2\}$ and $j\in\{0,1\}$. If $a\in A_i$ then $M_j(b,a)=1$ only if $b\in A_{i+j}$. Moreover, and each row or column of the matrix $M_j$ has precisely $q^2$ nonzero entries.
\end{lemma}
\begin{proof}
The first statement is clear from the definition of the transition matrices. We prove the second statement in the case $j=1$.
Choose $b=\Gamma\tau \in A$ and refer to Figure \ref{q2choices}, where $\tau$ is shaded.
There are precisely $q^2$ parallelograms $\pi$ of shape $(3,2)$ such that
  $t(\pi)=\tau$.
  For once $\tau$ is chosen, there are $q$ choices for
  the chamber $\delta_1$. Once $\delta_1$ is chosen, there are $q$ choices for $\delta_2$ and the whole parallelogram
  of shape $(3,2)$ is then completely determined. It follows that there are $q^2$ possibilities, as claimed.
  This proves that for each $b \in A$, there are $q^2$ choices for $a \in A$ such that $M_1(b,a)=1$. That is, each column of the matrix $M_1$ has precisely $q^2$ nonzero entries. A similar argument applies to rows.
  \refstepcounter{picture}
\begin{figure}[htbp]\label{q2choices}
\hfil
\centerline{
\beginpicture
\setcoordinatesystem units <0.4cm,0.7cm> 
\setplotarea x from -5 to 5, y from 0 to 5         
\setlinear \plot
 -1 1  0 0   2 2  0 4  -1 3  0 2  1 3 /
\plot -1 1  0 2  1 1  /
\plot  -1 1  -2 2  -1 3 /
\plot -1 1  1 1  / 
\plot  -2 2  2 2 /  
\plot -1 3  1 3  /   
\plot 0 4   1 5  2 4  0 4  1 3  2 4  3 3  1 3  2 2  3 3 /  
\setshadegrid span <1.5pt>
\vshade  -1  3  3  <,z,,>  1  1  5  <z,,,>  3  3  3  /
\put{{\small $\delta_1$}}   at  0 1.4
\put{{\small $\delta_2$}}   at  0 0.65
\arrow <10pt> [.3,.67] from   5.4  1.5   to     2.8  2.2
\put{$\tau$}  at 6  1.4
\endpicture
}
\caption{}
\end{figure}
\end{proof}

\section{K-theory}\label{K-theory}

The main result of \cite{rs3}, later extended to a more general class of $C^*$-algebras in \cite{ev}, is that the K-theory of a rank 2 Cuntz-Krieger algebra can be expressed in terms of the transition matrices $M_1, M_2$.
The matrix $\begin{smallmatrix}(I-M_1,&I-M_2)\end{smallmatrix}$ defines a homomorphism $\bb Z^A\oplus \bb Z^A \to \bb Z^A$. Let $r$ be the rank, and $T$ the torsion part, of the finitely generated abelian group
$C(\Gamma)=\coker\begin{smallmatrix}(I-M_1,&I-M_2)\end{smallmatrix}$. Thus
$C(\Gamma)\cong\bb Z^r\oplus T$.

\begin{theorem}\label{Ki}
Let $\Delta$ be a building of type $\widetilde A_2$ and order $q$, with maximal boundary $\Omega$.
Let $\Gamma$ be a group of type preserving automorphisms of $\Delta$ which acts regularly on the chambers of $\Delta$.
Then
\begin{equation}\label{K}
K_0(\fk A_\Gamma) = K_1(\fk A_\Gamma)= \bb Z^{2r}\oplus T.
\end{equation}
\end{theorem}

\begin{proof}
The equation (\ref{K}) is the result of combining \cite[Proposition 4.13]{rs3} with the analogue of \cite[Lemma 5.1]{rs3}, whose proof carries over without change to the larger tiles considered in the present article.
\end{proof}

In order to compute the K-theory of $\fk A_\Gamma$ it is convenient to reduce the size of the tiles.
A ball of radius one in an apartment is a hexagon containing six chambers.
A \textit{reduced tile} is a hexagon together with a choice of one of its six chambers, which determines the direction ``up''.
The reduced tile $\xi$ is said to be of type $i\in \bb Z_3$, if its central vertex $v_{11}(\xi)$ is of type $i$.
Denote by $X_i$ the set of reduced tiles of type $i$.
If $\xi \in X_i$ then its chambers are denoted $\xi_k, \overline\xi_k, k\in \bb Z_3$ as in Figure \ref{reduced tile},
with the ``up'' chamber denoted by $\xi_i$.
The vertices $v_{kl}(\xi)$ of type $k-l+i$ are labelled as in Figure \ref{reduced tile}.
This is consistent with previous notation, in the sense that if $\tau$ is any tile containing the hexagon $\xi$
then $v_{kl}(\xi)=v_{kl}(\tau)$ whenever both sides are defined.
Note that there are $q^2$ tiles $\tau$ containing $\xi$, since once the hexagon is fixed there are $q$ choices of $v_{00}(\tau)$
and $q$ choices of $v_{22}(\tau)$.

\begin{remark}
  The link of a vertex $v$ in $\Delta$ is the spherical building of a projective plane of order $q$.
   The hexagonal boundary of a reduced tile $\xi$ with central vertex $v$ is an apartment in this spherical building. The vertex $v_{ij}(\xi)$ is a point [line]  in the projective plane if $i-j=1$ $[i-j=2]$.
\end{remark}

\refstepcounter{picture}
\begin{figure}[htbp]
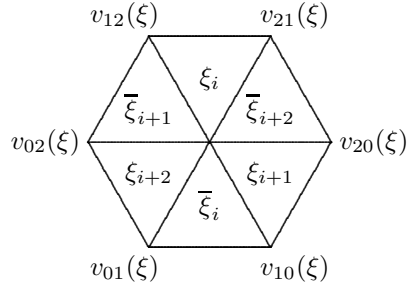
\label{reduced tile}
\hfil
\centerline{
\beginpicture
\setcoordinatesystem units <0.8cm,1.4cm>  point at -4 0 
\setplotarea x from -5 to 5, y from 1 to 4         
\setlinear
\plot -1 1  1 1   2 2  1 3  -1 3  0 2  1 3 /
\plot -1 1  0 2  1 1  /
\plot  -1 1  -2 2  -1 3 /
\plot -1 1  1 1  / 
\plot  -2 2  2 2 /  
\plot -1 3  1 3  /   
\put{$v_{10}(\xi)$}     at 1.4   0.8
\put{$v_{01}(\xi)$}     at -1.4  0.8
\put{$v_{20}(\xi)$}     at  2.7  2
\put{$v_{02}(\xi)$}     at -2.7  2
\put{$v_{21}(\xi)$}     at  1.4   3.2
\put{$v_{12}(\xi)$}     at -1.4   3.2
\put{$\xi_i$}       at  0  2.6
\put{$\overline \xi_i$}       at  0  1.4
\put{$\xi_{i+1}$}              at  1  1.7
\put{$\overline \xi_{i+1}$}    at -1  2.3
\put{$\xi_{i+2}$}              at  -1  1.7
\put{$\overline \xi_{i+2}$}    at  1  2.3
\endpicture
}
\hfil
\caption{A reduced tile of type $i$.}
\end{figure}

 Define a new alphabet $\Lambda$ to be the set of $\Gamma$-orbits of reduced tiles. Since the action of $\Gamma$ is type preserving,
we may write $\Lambda=\Lambda_1\cup\Lambda_2\cup\Lambda_3$ where $\Lambda_i=\{\Gamma\xi : \xi\in X_i \}$.
For each fixed chamber $\delta$, there are $q^3$ reduced tiles $\xi \in X_i$ such that $\xi_i= \delta$ \cite[Lemma 4.7]{RR}.
Since $\Gamma$ acts regularly on the set of chambers of $\Delta$, $\#\Lambda_i=q^3$ and $\#\Lambda=3q^3$.
Each reduced tile is contained in $q^2$ tiles, and so $\# A= 3q^5$.

The new transition matrices
$N_1, N_2$ are defined by the diagrams in Figure \ref{new transition}. In each of the diagrams,
the region is a union of two hexagons (one shaded, one partly shaded), which are representatives of the elements $a, b$ respectively.
\refstepcounter{picture}
\begin{figure}[htbp]
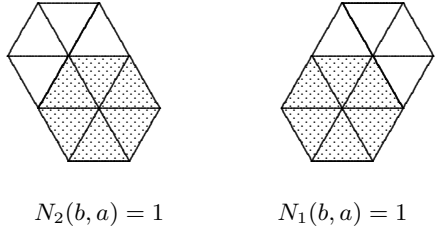
\label{new transition}
\hfil
\centerline{
\beginpicture
\setcoordinatesystem units <0.4cm,0.7cm>  point at -4 0 
\setplotarea x from -5 to 5, y from 1 to 4         
\setlinear
\plot -1 1  1 1   2 2  0 4  -1 3  0 2  1 3 /
\plot -1 1  0 2  1 1  /
\plot  -1 1  -2 2  -1 3 /
\plot -1 1  1 1  / 
\plot  -2 2  2 2 /  
\plot -1 3  1 3  /   
\plot 0 4  2 4  0 4  1 3  2 4  3 3  1 3  2 2  3 3 /  
\setshadegrid span <1.5pt>
\vshade  -2  2  2  <,z,,>  -1 1 3  <z,,,> 1 1 3  <z,,,>  2  2  2  /
\put{{\small $N_1(b,a)=1$}}   at  0 0
\setcoordinatesystem units <0.4cm,0.7cm>  point at 4 0 
\setplotarea x from -5 to 5, y from 2 to 4         
\setlinear
\plot -1 1  1 1  2 2  0 4  -1 3  0 2  1 3 /
\plot -1 1  0 2  1 1  /
\plot  -1 1  -2 2  -1 3 /
\plot -1 1  1 1  / 
\plot  -2 2  2 2 /  
\plot -1 3  1 3  /   
\plot  -2 4  0 4  -1 3  -2 4  -3 3  -1 3  -2 2  -3 3 /  
\setshadegrid span <1.5pt>
\vshade  -2  2  2  <,z,,>  -1 1 3  <z,,,> 1 1 3  <z,,,>  2  2  2  /
\put{{\small $N_2(b,a)=1$}}   at  0 0
\endpicture
}
\hfil
\caption{New transition matrices.}
\end{figure}
Once the shaded hexagon is chosen, there are $q^2$ choices
for the hexagon at the higher level which contains two shaded triangles. For there are $q$ choices for the unshaded triangle meeting the shaded hexagon
in its upper horizontal edge. Once this unshaded triangle is chosen there are $q$ choices for the unshaded triangle adjacent to it. The remaining two unshaded triangles are then completely determined, since they lie in the convex hull of the region consisting of the
shaded hexagon and the two chosen triangles.

\begin{lemma}\label{cokerN}
  $C(\Gamma)=\coker\begin{smallmatrix}(I-N_1,&I-N_2)\end{smallmatrix}$.
\end{lemma}

\begin{proof}
  This follows by applying successively parts (i) and (ii) of \cite[Lemma 6.1]{rs3} as indicated in Figure \ref{reduction}.
  The resulting $\{0,1\}$ matrices are precisely$N_1$ and $N_2$.
\end{proof}

\refstepcounter{picture}
\begin{figure}[htbp]
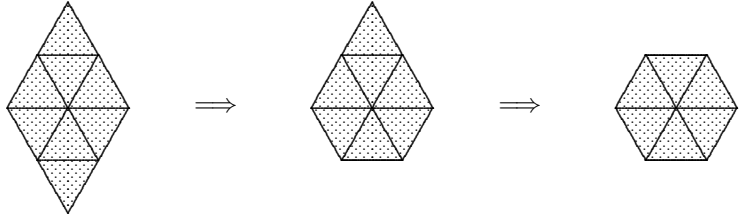
\label{reduction}
\hfil
\centerline{
\beginpicture
\setcoordinatesystem units <0.4cm,0.7cm> 
\setplotarea x from -5 to 5, y from 1 to 4         
\put{
\beginpicture    
\setlinear \plot
 -1 1  0 0   2 2  0 4  -1 3  0 2  1 3 /
\plot -1 1  0 2  1 1  /
\plot  -1 1  -2 2  -1 3 /
\plot -1 1  1 1  / 
\plot  -2 2  2 2 /  
\plot -1 3  1 3  /   
\setshadegrid span <1.5pt>
\vshade  -2  2  2  <,z,,>  0  0  4  <z,,,>  2  2  2  /
\endpicture
}
at -10 0
\put{$\Longrightarrow$} at  -5 2
\put{
\beginpicture    
\setlinear \plot
 -1 1  1 1   2 2  1 3  0 4  -1 3  0 2  1 3 /
\plot -1 1  0 2  1 1  /
\plot  -1 1  -2 2  -1 3 /
\plot -1 1  1 1  / 
\plot  -2 2  2 2 /  
\plot -1 3  1 3  /   
\setshadegrid span <1.5pt>
\vshade  -2  2  2  <,z,,>  -1 1 3  <z,z,,>  0  1  4  <z,z,,>   1  1  3 <z,,,>  2  2  2  /
\endpicture
}
at 0 0
\put{$\Longrightarrow$} at  5 2
\put{
\beginpicture  
\setlinear
\plot -1 1  1 1   2 2  1 3  -1 3  0 2  1 3 /
\plot -1 1  0 2  1 1  /
\plot  -1 1  -2 2  -1 3 /
\plot -1 1  1 1  / 
\plot  -2 2  2 2 /  
\plot -1 3  1 3  /   
\setshadegrid span <1.5pt>
\vshade  -2  2  2  <,z,,>  -1 1 3  <z,,,> 1 1 3  <z,,,>  2  2  2  /
\endpicture
}
at  10 0
\endpicture
} 
\hfil
\caption{Two-step reduction of a tile to a hexagon}
\end{figure}

\begin{remark}
  The $\{0,1\}$-matrices $N_1, N_2$ cannot be used to define the rank 2 Cuntz-Krieger algebra,
although they can be used to compute its K-theory. This is because the analogue of \cite[Lemma 7.4]{rs2} would fail.
\end{remark}

\begin{proposition}\label{basicK}
For $i=1,2$, we have  $K_i(\fk A_\Gamma)=\bb Z_{q^2-1}$.
\end{proposition}

\begin{proof}
Lemma \ref{cokerN} implies that, as a finitely presented abelian group, we have
\begin{equation} \label{idrelation}
C(\Gamma)=\big\langle \Lambda \, |\, a=\sum_{b\in \Lambda}N_j(b,a)b,\quad a\in \Lambda, j=1,2 \big\rangle.
\end{equation}
Fix $i\in \bb Z_3$ and $\xi\in X_i$, so that  $a=\Gamma\xi\in \Lambda_i$. Then $N_1(b,a)=1$ if and only if
$b=\Gamma\eta$ where $\eta \in X_{i+1}$,  $\eta_i=\xi_i$ and $\overline\eta_{i+1}=\overline\xi_{i+2}$.
Since $\Gamma$ acts regularly on the set of chambers,
\begin{equation}\label{shift1}
a =  \displaystyle \sum_
{
\substack{
\eta \in X_{i+1}\\
\eta_i=\xi_i \\
\overline\eta_{i+1}=\overline\xi_{i+2}
}
}
\Gamma\eta.
\end{equation}
Similarly
\begin{equation}\label{shift2}
a =  \displaystyle \sum_
{
\substack{
\zeta \in X_{i+2}\\
\zeta_i=\xi_i \\
\overline\zeta_{i+2}=\overline\xi_{i+1}
}
}
\Gamma\zeta.
\end{equation}

Equation (\ref{shift1}) implies that $a$ depends only on the chambers $\xi_i$ and $\overline\xi_{i+2}$ of $\xi$.
Equation (\ref{shift2}) implies that $a$ depends only on the chambers $\xi_i$ and $\overline\xi_{i+1}$ of $\xi$.
Therefore $a$ depends only on $\xi_i$.
More precisely, if $\vartheta\in X_i$ is any hexagon with $\vartheta_i=\xi_i$, then $\Gamma\vartheta=\Gamma\xi$
in $C(\Gamma)$. To see this, choose a hexagon $\varphi\in X_i$ with $\varphi_i=\xi_i$, $v_{02}(\varphi)=v_{02}(\vartheta)$
and $v_{20}(\varphi)=v_{20}(\xi)$. Then $\Gamma\varphi=\Gamma\xi$, by (\ref{shift1}), and $\Gamma\varphi=\Gamma\vartheta$, by (\ref{shift2}), so that $\Gamma\vartheta=\Gamma\xi$, as claimed.
However, $\Gamma$ acts transitively on chambers.
Therefore all elements of $\Lambda_i$ are equal, to $a_i$, say. Thus

\begin{equation*}
\begin{split}
  C(\Gamma) & = \langle a_0, a_1, a_2 \, | \, a_i=q^2a_{i+1}, a_i=q^2a_{i-1}, i\in\bb Z_3 \rangle  \\
   &= \langle a_0 \, | \, a_0=q^4a_0, a_0=q^6a_0 \rangle \\
   &= \langle a_0 \, | \, a_0=q^2a_0 \rangle = \bb Z_{q^2-1}.
\end{split}
\end{equation*}
In particular, since $C(\Gamma)$ is finite, it follows from Theorem \ref{Ki} that for $i=0,1$,
\begin{equation}
K_i(\fk A_\Gamma)=\bb Z_{q^2-1}.
\end{equation}
\end{proof}

Recall that the classical Cuntz algebra $\cl O_n$ is generated by $n$ isometries whose range projections sum to the identity operator.

\begin{corollary}
  The algebra $\fk A_\Gamma$ is stably isomorphic to $\cl O_{q^2}\otimes\cl O_{q^2}$.
\end{corollary}

\begin{proof}
$K_0(\cl O_n)= \bb Z_{n-1}$ and $K_1(\cl O_n)= 0$ \cite{ca}.
The K\"unneth Theorem for tensor products \cite[Theorem 23.1.3]{bl} shows that $K_i(\cl O_{q^2}\otimes\cl O_{q^2})= \bb Z_{q^2-1}$, $i=0,1$.
Since the algebras involved are p.i.s.u.n. and satisfy the U.C.T. \cite[Remark 6.5]{rs2}, the result follows from the classification theorem for such algebras \cite{k}.
\end{proof}

\begin{remark}\label{tensor}
In order to determine the isomorphism class of $\fk A_\Gamma = C(\Omega) \rtimes \Gamma$ we need to use that fact that it is
classified (in the class of p.i.s.u.n. $C^*$-algebras satisfying the U.C.T.) by the invariants $(K_0(\fk A_\Gamma), [\1], K_1(\fk A_\Gamma))$ as abelian groups with distinguished element in $K_0$ \cite{k}.
It turns out that  $\fk A_\Gamma$ is not isomorphic to $\cl O_{q^2}\otimes\cl O_{q^2}$, but to
$M_r  \otimes\cl O_{q^2}\otimes\cl O_{q^2}$, for some $r$. In order to determine $r$, we use the fact that,
for any $C^*$-algebra $\fk C$, the identity matrix $I_{M_r \otimes\fk C}$ is a direct sum of $r$ copies of $I_\fk C$, so that the class of the identity element in $K_0$ satisfies $[\1_{M_r \otimes\fk C}]=r\cdot[\1_\fk C]$.
Now, for any decorating set $E$, the $C^*$-algebra $\cl A_E$ is stably isomorphic to $\cl A_D$ \cite [Corollary 5.15]{rs2} and so
$K_i(\cl A_E)=\bb Z_{q^2-1}$. Since $q^2\cdot[\1]=[\1]$ in $K_0$, it follows from the classification theorem that
$M_{q^2}  \otimes \cl A_E \cong \cl A_E$.
\end{remark}

\begin{lemma}\label{mk}
Let $k=(q+1)(q^2+q+1)$, the number of chambers of $\Delta$ which contain the vertex $O$.
\begin{itemize}
\item   The decoration $\delta : D \to A$, defined by $\delta (\tau) = \Gamma\tau$, is equivalent to the decoration $\delta_0': A_0 \times \{1,2,\dots ,k\}\to A$, defined by
$\delta'((a_0,i)) = a_0$.
\item  $\cl A_D \cong \cl A_{A_0 \times \{1,2,\dots ,k\}}$.
\end{itemize}
\end{lemma}

\begin{proof}
Let $c_1, \dots , c_k$ be the chambers in $\Delta$ based at $O$.
For each $i$, let $D_i$ denote the set of tiles $\tau\in D$ which have initial chamber $c_i$.
Since $\Gamma$ acts freely and transitively on the set of chambers of $\Delta$, the restriction $\delta_i$
of $\delta$ to $D_i$ is a bijection onto $A_0$.
Define $\eta : D \to A_0 \times \{1,2,\dots ,k\}$ by $\eta (\tau)=(\delta_i(\tau), i)$ for $\tau\in D_i$,
$i= 1, 2, \dots, k$. Then $\eta$ is bijective and $\delta_0'\eta=\delta$. Thus $\delta$ and $\delta_0'$ are equivalent decorations.

The fact that $\cl A_D \cong \cl A_{A_0 \times \{1,2,\dots ,k\}}$ follows immediately from Remark \ref{decorate}.
\end{proof}

\begin{lemma}\label{A0}
  $\fk A_\Gamma \cong  M_{3(q+1)}  \otimes \cl A_{A_0}$.
\end{lemma}

\begin{proof}
  By Proposition \ref{GammaD}, Lemma \ref{mk}, Lemma \ref{decorate*}, and  Remark \ref{tensor},
  $$\fk A_\Gamma \cong \cl A_D \cong \cl A_{A_0 \times \{1,2,\dots ,k\}}\cong M_k\otimes \cl A_{A_0}\cong  M_{3(q+1)}  \otimes \cl A_{A_0},$$
  since $k \equiv 3(q+1) \pmod {q^2-1}$.
\end{proof}

\begin{lemma}\label{3times}
  $\cl A_A \cong M_3  \otimes \cl A_{A_0}$.
\end{lemma}

\begin{proof}

For $i=0,1,2$, let $e_i=\sum_{a\in A_i}s_{a,a}$, so that $e_0+e_1+e_2=\1$.
It is an immediate consequence of relations (\ref{rel1b*}), (\ref{rel1d*}) that the generators of $\cl A_A$ satisfy
\begin{equation*}
e_is_{u,v}e_i=
\begin{cases}
    s_{u,v} &  \text{if $o(u), o(v) \in A_i$},\\
    0 &  \text{otherwise}.
  \end{cases}
\end{equation*}
It follows immediately from the uniqueness statement of Theorem \ref{mainck} that
\begin{equation}\label{Ai}
  e_i\cl A_A e_i \cong \cl A_{A_i}.
\end{equation}

The relations (\ref{rel1*}) show that if $M_j(b,a)=1$ then the partial isometry
$s_{b,ab}$ has initial projection $s_{ab,ab}$ and final projection $s_{b,b}$.
Thus $[s_{ab,ab}]=[s_{b,b}]$. It follows from (\ref{rel1c*}) and Lemma \ref{01} that, if $a\in A_i$ then
$[s_{a,a}] = \sum_{b\in A_{i+1}}M_1(b,a)[s_{b,b}]$. Therefore
\begin{equation*}
\begin{split}
[e_i] &  = \sum_{a\in A_i}[s_{a,a}] =  \sum_{a\in A_i}\sum_{b\in A_{i+1}}M_1(b,a)[s_{b,b}]\\
 & =  \sum_{b\in A_{i+1}}\large\left(\sum_{a\in A_i}M_1(b,a)\large\right)[s_{b,b}]
 = q^2[e_{i+1}]
\end{split}
\end{equation*}
since $\sum_{a\in A_i}M_1(b,a)=q^2$ for each $b\in A_{i+1}$, by Lemma \ref{01}.
Similarly (replacing $M_1$ by $M_2$), $[e_i]=q^2[e_{i-1}]$.
It follows that $[e_i]=q^4[e_i]=q^6[e_i]$. Therefore $[e_i]=q^2[e_i]$ and $[e_i]=[e_{i+1}]$, for all $i\in \bb Z/3\bb Z$.
Thus $[e_1]=[e_2]=[e_3]$ in $K_0(\cl A_A)$.
In view of the description of the K-theory of purely infinite $C^*$-algebras in \cite[Section 1]{ca},
this implies that the projections $e_0, e_1, e_2$ are Murray-von Neumann equivalent in $\cl A_A$.
It follows from (\ref{Ai}) that $\cl A_{A_0}\cong \cl A_{A_1}\cong \cl A_{A_2}$ and that
$\cl A_A \cong M_3  \otimes \cl A_{A_0}$.
\end{proof}

\begin{lemma}\label{basic}
  $(K_0(\cl A_A), [\1]) \cong (\bb Z_{q^2-1}, 3q)$.
\end{lemma}

\begin{proof}
Since $\cl A_A$ is stably isomorphic to $\fk A_\Gamma$, it follows from Proposition \ref{basicK} that
 $K_0(\cl A_A)\cong K_0(\fk A_\Gamma) \cong C(\Gamma) \cong \bb Z_{q^2-1}$.
By the proof of \cite [Proposition 8.3]{rs3},
the isomorphism between $K_0(\cl A_A)$ and the group
\begin{equation}\label{star}
\big\langle A \, |\, a=\sum_{b\in A}M_j(b,a)b,\quad a\in A, j=1,2 \big\rangle.
\end{equation}
maps $[\1]$ to the element $\sum_{a\in A} a$.
However $\# A=3q^5 \equiv 3q \pmod {q^2-1}$. The result follows.
\end{proof}

\begin{theorem}
  $(K_0(\fk A_\Gamma), [\1], K_1(\fk A_\Gamma))\cong (\bb Z_{q^2-1},3(q+1) ,\bb Z_{q^2-1})$.
\end{theorem}

\begin{proof}
By Lemmas \ref {A0} and \ref{3times},
$$\fk A_\Gamma \cong  M_{3(q+1)}  \otimes \cl A_{A_0} \cong  M_{q+1} \otimes M_3 \otimes \cl A_{A_0}
  \cong  M_{q+1} \otimes \cl A_A.$$
  It follows from Lemma \ref{basic} that $[\1_\Gamma]=(q+1)\cdot 3q = 3(q+1)$ in $\bb Z_{q^2-1}$.
\end{proof}

\noindent In view of the classification theorem \cite{k} and Remark \ref{tensor}, this completes the proof of Theorem \ref{mainboundaryaction}.

\begin{remark}
  The order of $[\1]$ in $K_0(\fk A_\Gamma)$ is
  \begin{equation*}
\begin{split}
\frac{q^2-1}{(q^2-1,3(q+1))}& = \frac{q-1}{(q-1,3)}\\
 & =
\begin{cases}
    q-1&  \text{if $q \not\equiv 1 \pmod {3}$},\\
    \frac{q-1}{3} &  \text{if $q  \equiv 1 \pmod {3}$}.
  \end{cases}
\end{split}
\end{equation*}
Compare this with the conjecture in \cite[Remark 8.4]{rs3}.
\end{remark}

There are other ways of expressing $\fk A_\Gamma$ in terms of rank-1 Cuntz-Krieger algebras, in addition to that of Theorem \ref{mainboundaryaction}. Here is an example. Let $\cl A_0=\cl O_A$, where
\begin{equation*}
A=
\begin{pmatrix}
1&0&1\\
0&1&1\\
1&1&0
\end{pmatrix}.
\end{equation*}
Then $\cl A_0$ is a simple rank-1 Cuntz-Krieger algebra  such that $K_*(\cl A_0)=(\mathbb Z,\mathbb Z)$.

\begin{proposition}
Under the hypotheses of Theorem \ref{mainboundaryaction}, $\fk A_\Gamma$ is stably isomorphic to $\cl A_0 \otimes \cl O_{q^2}$.
\end{proposition}

\begin{proof}
The K\"unneth Theorem for tensor products \cite[Theorem 23.1.3]{bl} shows that
$K_*(\cl A_1 \otimes \cl A_2)= (\bb Z_{q^2-1},\bb Z_{q^2-1})$.
Since the algebras involved are all p.i.s.u.n. and satisfy the U.C.T., the result follows from the Classification Theorem \cite{ad}.
\end{proof}

\end{document}